\documentclass[11 pt]{amsart}
\usepackage{amsmath,enumerate,amsthm,amssymb,amscd}

\newcommand{\pd}{\operatorname{pd}}

\newcommand{\Hom}{\operatorname{Hom{}}}
\newcommand{\Ext}{\operatorname{Ext{}}}
\newcommand{\Tor}{\operatorname{Tor{}}}
\newcommand{\lc}{\operatorname{H}}

\newcommand{\depth}{\operatorname{depth}}

\renewcommand{\phi}{\varphi}
\renewcommand{\to}{{\longrightarrow}}

\newcommand{\Depth}{\operatorname{Depth}}
\newcommand{\rG}{\tilde{\text{\rm G}}}
\newcommand{\Gdim}{\operatorname{Gdim}}
\newcommand{\rGdim}{\operatorname{\tilde{\text{\rm G}}dim}}
\newcommand{\fp}[1]{(FP)_{#1}}
\newcommand{\ov}[1]{\overline{#1}}
\newcommand{\chk}{\text{v}}
\newcommand{\fpi}{(FP)_{\infty}}
\newcommand{\fpid}{\operatorname{FP-id}}

\newcommand{\directlimit}[2]{\ensuremath{\varinjlim_{#1} #2}}

\newtheorem{thm}{Theorem}
\newtheorem{cor}[thm]{Corollary}
\newtheorem{prop}[thm]{Proposition}
\newtheorem{lemma}[thm]{Lemma}
\newtheorem{defn}[thm]{Definition}
\newtheorem{remark}[thm]{Remark}

\newtheorem{example}[thm]{Example}

\numberwithin{equation}{section}

\begin{document}

\title[The Auslander-Bridger formula for coherent rings]{The Auslander-Bridger formula and the Gorenstein property for coherent rings}

\author {Livia Hummel}

\address{Department of
Mathematics and Computer Science\\
University of Indianapolis\\  Indianapolis, IN 46277}
\email{ hummell@uindy.edu}

\author{Thomas Marley}

\address{Department of
Mathematics\\
University of Nebraska-Lincoln\\Lincoln,  NE 68588-0130}
\email{ tmarley@math.unl.edu}

\date{February  6, 2009}

\bibliographystyle{amsplain}

\numberwithin{thm}{section}

\begin{abstract} The concept of  Gorenstein dimension, defined by Auslander and Bridger for finitely generated modules over a Noetherian ring, is studied in the context of  finitely presented modules over a coherent ring.  A generalization of the Auslander-Bridger formula is established and is used as a cornerstone in the development of a theory of coherent Gorenstein rings.
\end{abstract}

\subjclass[2000]{ Primary
13D05, 13D07 }
\keywords{Gorenstein dimension, coherent ring}

\maketitle

\section{Introduction}   In addressing a problem posed by Glaz (\cite{Glaz1992a}, \cite{Glaz1994a}),  Hamilton and the second author  give a definition of Cohen-Macaulay for commutative rings which agrees with the usual notion for Noetherian rings with the property that every coherent regular ring is Cohen-Macaulay \cite{Hamilton2007a}.   (A quasi-local ring is defined to be {\it regular} if every finitely generated ideal has finite projective dimension.)    A natural question is whether there is a reasonable concept of Gorenstein for commutative rings  such that every
coherent regular ring is Gorenstein and every coherent Gorenstein ring is Cohen-Macaulay.    In this paper, we develop such a theory of coherent Gorenstein rings which mirrors much of the theory in the Noetherian case.  Central to this development is the concept of Gorenstein dimension (G-dimension, for short), first introduced in the context of finitely generated modules over Noetherian rings by  Auslander and  Bridger \cite{Auslander1969a}.   In particular, we prove the following generalization of the Auslander-Bridger formula for coherent rings using a notion of depth for arbitrary quasi-local rings developed by Barger \cite{Barger1972a}, Hochster \cite{Hochster1974a}, and Northcott \cite{Northcott1976a}:

\begin{thm} \label{thmA} Let $R$ be a quasi-local coherent ring and $M$ an $R$-module of finite G-dimension.  Then
$$\depth M+\Gdim_RM=\depth R.$$
\end{thm}

Since working with depth in the non-Noetherian case typically requires passage to faithfully flat extensions (which does not generally preserve coherence), our development of G-dimension must be done with some care.  Out of necessity, we prove a version of Theorem \ref{thmA} which holds for modules over an arbitrary commutative ring which have free resolutions consisting of finitely generated free modules in each degree (so-called $\fpi$ modules).

Auslander and Bridger prove that every ideal having finite G-dimension characterizes Noetherian local Gorenstein rings.  It is natural, therefore, to make the following definition:

\begin{defn}{\rm A quasi-local ring $R$ is said to be {\it Gorenstein} if every finitely generated ideal has finite G-dimension.}
\end{defn}

We are able to prove several properties of coherent Gorenstein rings which are analogous to results in the Noetherian case.  We summarize some of these results in the following:

\begin{thm} \label{thmB}  Let $R$ be a quasi-local coherent ring.
\begin{enumerate}
\item  If $x$ is a non-unit non-zero-divisor on $R$, then $R$ is Gorenstein if and only if $R/(x)$ is Gorenstein.
\item If $x$ is an indeterminate over $R$, then $R$ is Gorenstein if and only if $R[x]$ is Gorenstein.
\item If $\depth R<\infty$, then $R$ is Gorenstein if and only if $R$ has finite FP-injective dimension.
\item Every coherent regular ring is Gorenstein.
\item Every coherent Gorenstein ring is Cohen-Macaulay.
\item If $R$ is a coherent quasi-local Gorenstein ring with $\depth R=n$, and $x_1,\dots,x_n$ is a regular sequence on $R$, then $(x_1,\dots,x_n)$ is an irreducible ideal.
\end{enumerate}
\end{thm}

The paper is structured in the following way:  In section two, we summarize the main results we need concerning non-Noetherian depth and $\fpi$ modules.   Section three defines G-dimension and restricted G-dimension and proves their basic properties.  Section four is devoted to the proof of the Auslander-Bridger formula for modules of finite restricted G-dimension, while in section five the main results concerning Gorenstein rings are established.

The authors would like to thank thank the following people for their many helpful comments and suggestions relating to  this work:  Lucho Avramov, Lars Christensen, Sri Iyengar, and Diana White.   The authors are also grateful to the referee for suggesting Example \ref{noncohGor}.

\section{Preliminaries}

Throughout  $R$ will always denote a commutative ring with identity.  All modules are assumed to be unital.  The term {\it local ring} will be used exclusively for commutative Noetherian rings with a unique maximal ideal.  When the ring is not necessarily Noetherian, the term {\it quasi-local} will be used.  A ring is {\it coherent} if every finitely generated ideal is finitely presented.  See \cite{Glaz1989a} for basic properties of coherent rings.

The following definition is due to Bieri \cite{Bieri1976a}:

\begin{defn} An $R$-module $M$ is said to be $(FP)_n^R$ for some $n\ge 0$ if there exists an exact sequence $$F_n\to F_{n-1}\to \cdots \to F_1\to F_0 \to M\to 0$$
where  $F_0,\dots, F_n$ are finitely generated free $R$-modules.  For brevity we often write $\fp{n}$ for $(FP)_n^R$ when there is no ambiguity about the ring $R$.   If $M$ is $\fp{n}$ for all $n\ge 0$ then we say $M$ is $\fpi$.
\end{defn}

We note that over a Noetherian ring the class of \(\fpi\)-modules coincides with the class of  of finitely generated modules.  Over a coherent ring,  \(\fpi\)-modules are just the finitely presented modules (cf.\cite[Corollary 2.5.2]{Glaz1989a}).

\begin{lemma} Let $M$ be an $R$-module and $n\ge 0$ an integer.  The following are equivalent:
\begin{enumerate}
\item $M$ is $\fp{n+1}$.
\item The functors $\Ext^i_R(M,-)$ commute with direct limits for $i=0,\dots,n$.
\end{enumerate}
\end{lemma}

\begin{proof} See \cite[Corollary 1.6]{Bieri1976a}.
\end{proof}

\begin{lemma} \label{fpi}  Let  $0\to L\to M\to N\to 0$ be a short exact sequence of $R$-modules.  Then the following hold for any $n\ge 0$:
\begin{enumerate}
\item  If $L$ is $\fp{n}$ and $M$ is $\fp{n+1}$ then $N$ is $\fp{n+1}$.
\item If $M$ and $N$ are $\fp{n+1}$ then $L$ is $\fp{n}$.
\item If $L$ and $N$ are $\fp{n}$ then $M$ is $\fp{n}$.
\end{enumerate}
Consequently, if any two modules in a short exact sequence are $\fpi$ then so is the third.
\end{lemma}

\begin{proof} See \cite[Proposition 1.4]{Bieri1976a} or \cite[Theorem 2.1.2]{Glaz1989a}.
\end{proof}

The following remark is easily seen from the definition of $\fp{n}$ and Lemma \ref{fpi}:

\begin{remark} \label{fpi-flat} Let $S$ be a flat $R$-algebra, $M$ an $R$-module, and $n\ge 0$ an integer.
\begin{enumerate}
\item If $M$ is $\fp{n}^R$ then $M\otimes_RS$ is $\fp{n}^S$.
\item If $S$ is faithfully flat then the converse to (1) holds.
\end{enumerate}
\end{remark}

We also will need the following:

\begin{remark}\label{fpi-nzd} Suppose $M$ is $\fpi^R$ and $x\in R$ is a non-zero-divisor on $R$ and $M$.
Then $M/xM$ is $\fpi^{R/(x)}$.
\end{remark}
\begin{proof}  Let $\mathbf F$ be a free resolution of $M$ consisting of finitely generated free $R$-modules.  As $x$ is a non-zero-divisor on $R$ and $M$, $\operatorname{Tor}^R_i(M,R/(x))=0$ for all $i>0$.  Hence, $\mathbf F\otimes_RR/(x)$ is a resolution of $M/xM$ consisting of finitely generated free $R/(x)$-modules.
\end{proof}

Depth plays an important role in the theory of G-dimension over local rings.   The notion of depth was extended to quasi-local rings by Hochster \cite{Hochster1974a}:

\begin{defn}{\rm  Let $M$ be an $R$-module and $I$ an ideal of $R$ such that $IM\neq M$.  Define the {\it classical depth} of $M$ with respect to $I$, denoted $\Depth_IM$, to be the supremum of the lengths of regular sequences on $M$ contained in $I$.  Define the {\it depth} of $M$ with respect to $I$, denoted $\depth_IM$, to be the supremum of $\Depth_{IS} M\otimes_RS$ over all faithfully flat extensions $S$ of $R$.  If $R$ is quasi-local with maximal ideal $m$, we denote $\depth_mM$ by $\depth_R M$ or simply $\depth M$.}
\end{defn}

We note that if $R$ is Noetherian and $M$ is finitely generated then $\depth_IM=\Depth_IM$ for all ideals $I$ of $R$ (see, for instance, part (5) of Proposition \ref{depth-basic} below).  We summarize some basic properties of $\depth$ in the following Proposition.  We refer the reader to \cite{Barger1972a}, \cite{Hochster1974a}, \cite[Chapter 5]{Northcott1976a}, \cite[Chapter 7]{Glaz1989a} or \cite[Section 9.1]{Bruns1993a} for details.

\begin{prop} \label{depth-basic} Let $M$ be an $R$-module and $I$ an ideal of $R$ such that $IM\neq M$.
\begin{enumerate}
\item $\depth_IM=\sup \{\depth_JM\mid J\subseteq I, J \text{ finitely generated}\}$.
\item If $I=(x_1,\dots,x_n)$ then $\depth_IM=\inf \{i\ge 0\mid H_{n-i}(\mathbf x,M)\neq 0\}$, where
$H_j(\mathbf x, M)$ denotes the $j$th Koszul homology of $\mathbf x=x_1,\dots,x_n$ on $M$.
\item $\depth_IM=\depth_{IS}(M\otimes_RS)$ for any faithfully flat $R$-algebra $S$.
\item If  $\depth_IM>0$ then $\Depth_{IS}M\otimes_RS>0$ where $S=R[X]$, a polynomial ring in one variable over $R$.
\item If $I$ is generated by $n$ elements then $\depth_IM=\Depth_{IS}M\otimes_RS$ where
$S=R[X_1,\dots,X_n]$, a polynomial ring in $n$ variables over $R$.
\item $\depth_IM=\depth_{\sqrt{I}}M$.
\item If $x\in I$ is $M$-regular then $\depth_IM=\depth_IM/xM +1$.
\item Suppose $0\to L\to M\to N\to 0$ is a short exact sequence of $R$-modules such that $IL\neq L$ and $IN\neq N$.  If $\depth_IM> \depth_IN$ then $\depth_IL=\depth_IN+1$.
\end{enumerate}
\end{prop}

In the case $I$ is $\fpi$ we have another characterization of $\depth_IM$:

\begin{lemma} \label{depth-ext} Let $M$ be an $R$-module and  $I$ an ideal such that $IM\neq M$.  Suppose $R/I$ is $\fp{n}$.  The following are equivalent:
\begin{enumerate}
\item $\depth_IM\ge n$.
\item $\Ext^i_R(R/I,M)=0$ for $0\le i< n$.
\end{enumerate}
\end{lemma}
\begin{proof} See Theorem 7.1.2 or Theorem 7.1.8 of \cite{Glaz1989a}.
\end{proof}

Combining  Lemma \ref{depth-ext} with part (1) of Proposition \ref{depth-basic}, we have:

\begin{prop} \label{depth-coherent}  Let $R$ be a coherent quasi-local ring with maximal ideal $m$ and $M$ a nonzero $R$-module such that $mM\neq M$. Then
$$\depth M:=\sup\{n\ge 0\mid \Ext^i_R(R/I,M)=0 \text{ for all }i<n \text{ for some f.g. ideal }I\subseteq m\}.$$
\end{prop}

\bigskip

\section{Gorenstein dimension}

As noted in the introduction, the theory of Gorenstein dimension, or G-dimension,  for finitely generated modules over a Noetherian ring was developed by Auslander and Bridger in \cite{Auslander1969a}.   This theory was later generalized for modules over arbitrary rings (Gorenstein projective dimension) by Enochs and Jenda \cite{Enochs1995a}.  However, for the notion of depth to be used effectively (e.g., as an inductive tool), one needs to restrict to a resolving class of finitely generated modules, so that Nakayama's lemma can be applied.  Hence, our development of G-dimension falls somewhere between Auslander-Bridger's and Enochs-Jenda's.    While our main interest is the case of finitely presented modules over coherent rings, restricting the theory to such rings does not allow one to fully utilize the concept of depth, which requires passage to faithfully flat extensions to guarantee the existence of non-zero-divisors.  (It is well-known that coherence is not preserved under faithfully flat extensions, even finitely generated ones.  See  \cite[Example 7.3.13]{Glaz1989a}, for example.)

For an $R$-module $M$ we let $M^*$ denote the dual module $\Hom_R(M,R)$.  Following Auslander and Bridger \cite{Auslander1969a}, we make the following definition:

\begin{defn}{\rm A finitely generated $R$-module $M$ is said to be a member of the {\it G-class} of $R$, denoted $\text{\rm G}(R)$, if the following hold:
\begin{enumerate}
\item $\Ext^i_R(M,R)=0$ for $i>0$.
\item $\Ext^i_R(M^*,R)=0$ for $i>0$.
\item The natural map $M\to M^{**}$ is an isomorphism.
\end{enumerate}}
\end{defn}

The class $\text{\rm G}(R)$ is closed under extensions and summands.  In particular, every finitely generated projective module is a member of $\text{\rm G}(R)$.  However, if $R$ is not Noetherian, $\text{\rm G}(R)$ may not be closed under duals (since the dual of a member of $\text{\rm G}(R)$ may not be finitely generated -- see Example \ref{noncohGor}).  A much better behaved class for our purposes  is the following:

\begin{defn}{\rm An $R$-module $M$ is a member of the {\it restricted G-class} of $R$, denoted
$\rG(R)$, if the following hold:
\begin{enumerate}
\item $M$ is in $\text{\rm G}(R)$.
\item $M$ is $\fpi$.
\item $M^*$ is $\fpi$.
\end{enumerate}}
\end{defn}

The class $\rG(R)$ is easily seen to be closed under extensions, summands, and duals.  Additionally, every finitely generated projective $R$-module is a member of $\rG(R)$.  Since any reflexive module is isomorphic to a submodule of a free module, we note that modules in \(G(R)\) and \(\rG(R)\) are torsion-free.

Let $M$ be an $R$-module.  A complex $\mathbf G$ is called a {\it G-resolution} (respectively, {\it $\rG$-resolution}) of $M$ if each $G_i$ is a member of the G-class (respectively, $\rG$-class), $G_n=0$ for $n<0$, $H_i(\mathbf G)=0$ for all $i\neq 0$, and $H_0(\mathbf {G})\cong M$.  For  a resolution $\mathbf G$ we set the {\it length} of $\mathbf G$ to be the  infimum of the set of integers $n$ such that $G_n\neq 0$. (By convention, we define the length of the zero complex to be zero.)  For an integer $n$ we define  the $n$th {\it syzygy} of $\mathbf G$  to be the kernel of the map $G_{n-1}\to G_{n-2}$ (here we let $G_{-1}=M$). If $M$ has a G-resolution (respectively, $\rG$-resolution), we define the G-dimension of $M$, $\Gdim_RM$ (respectively, the $\rG$-dimension $\rGdim_RM$), to be the infimum of the lengths of all G-resolutions (respectively, $\rG$-resolutions) of $M$.  Otherwise, we say the G-dimension (respectively, $\rG$-dimension) of $M$ is undefined.

We remark that $\Gdim_R M$ and $\rGdim_RM$ are defined for any $R$-module  $M$ which is  $\fpi$, since in this case $M$ has a resolution consisting of finitely generated free $R$-modules.  We also note that any syzygy module of a G-resolution (respectively, $\rG$-resolution) is finitely generated (respectively, $\fpi$). Since every module in $\rG(R)$ is a member of $\text{\rm G}(R)$, it is clear that $\Gdim_RM\le \rGdim_RM$ for every $\fpi$ $R$-module $M$.  We'll show below that if $\rGdim_RM<\infty$ then $\Gdim_RM=\rGdim_RM$ (Corollary \ref{G=rGa}).  Furthermore, if $R$ is coherent and $M$ is finitely presented then $\Gdim_RM=\rGdim_RM$ (Corollary \ref{G=rGb}).

In the remainder of this section we summarize  the basic properties of $\Gdim$ and $\rGdim$.  The proofs of many of these results  are similar to the analogous results for finitely generated modules over a Noetherian ring, once one observes the relevant modules are finitely generated or $\fpi$.  In a few instances alternative arguments must be made.    We refer the reader to Chapter 1 of \cite{Christensen2000a} for a clear and thorough treatment of G-dimension over Noetherian rings.  More complete proofs of the results below can be found in \cite{Miller2008a}.

\begin{prop} \label{Gclass-ses} Let $0\to L\to M\to N\to 0$ be a short exact sequence of $R$-modules.  Suppose $N$
is a member of the G-class (respectively, $\rG$-class) and $L$ and $M$ are finitely generated (respectively, $\fpi$).  Then $L\in G(R)$ if and only if $M\in G(R)$ (respectively, $L\in \rG(R)$ if and only if $M\in \rG(R)$).
\end{prop}

\begin{proof}  See \cite[Lemma 1.1.10(a)]{Christensen2000a}.  We add the observation that, as
$0\to N^*\to M^*\to L^*\to 0$ is exact and $N^*$ is $\fpi$, we have $L^*$ is $\fpi$ if and only if $M^*$ is $\fpi$ by Lemma \ref{fpi}.
\end{proof}

\begin{lemma} \label{Gdim=0} Let $M$ be an $R$-module of finite G-dimension (respectively, $\rG$-dimension) such that $\Ext^i_R(M,R)=0$ for all $i>0$.  Then $\Gdim_RM=0$ (respectively, $\rGdim_RM=0$).
\end{lemma}

\begin{proof} See \cite[Lemma 1.2.6]{Christensen2000a}.
\end{proof}

The following proposition will be used frequently:

\begin{prop} \label{GDProp} Let $M$ be an $R$-module which is $\fpi$ and $n\ge 0$ an integer.  The following are equivalent:
\begin{enumerate}
\item $\Gdim_RM\le n$.
\item $\Gdim_RM<\infty$ and $\Ext^i_R(M,R)=0$ for $i>n$.
\item  The $n$th syzygy of any G-resolution of $M$ is in the G-class of $R$.
\end{enumerate}
The same statement also holds if $\Gdim$, G-resolution, and G-class are replaced with $\rGdim$,
$\rG$-resolution, and $\rG$-class, respectively.
\end{prop}
\begin{proof}  The proof is similar to that of  \cite[Theorem 1.2.7]{Christensen2000a}, substituting Proposition \ref{Gclass-ses} and Lemma \ref{Gdim=0} where appropriate.  We add the observation that,  using the notation of the proof of  \cite[Theorem 1.2.7]{Christensen2000a}, $H_n$ is $\fpi$ if and only if $K_n$ is $\fpi$.
\end{proof}

Consequently, we have the following:

\begin{cor} \label{G=rGa}Let $M$ be a nonzero $R$-module of finite G-dimension.  Then
$$\Gdim_RM=\sup\{i\ge 0\mid \Ext^i_R(M,R)\neq 0\}.$$
The same statement also holds if G-dimension is replaced by $\rG$-dimension.  Hence, if $M$ has finite $\rG$-dimension then $\Gdim_RM=\rGdim_RM$.
\end{cor}

We note that if $R$ is coherent and $M$ is a finitely presented $R$-module then $M^*$ is also finitely presented (and hence $\fpi$).  This follows from the more general fact that, over a coherent ring, kernels and cokernels of homomorphisms between finitely presented modules are again finitely presented (cf.
Chapter 2 of \cite{Glaz1989a}).  In light of this observation, we  have:

\begin{cor} \label{G=rGb} Let $R$ be a coherent ring and $M$ a finitely presented $R$-module.  Then $\Gdim_RM=\rGdim_RM$.
\end{cor}
\begin{proof}  Note that as $R$ is coherent and $M$ is finitely presented, $M$ is $\fpi$.  Hence $\Gdim_RM$ and $\rGdim_RM$ are defined.
If $\Gdim_RM=\infty$ then clearly $\rGdim_RM=\infty$.  Suppose $\Gdim_RM=n<\infty$.   Let $\mathbf G$ be a $\rG$-resolution of $M$. By Proposition \ref{GDProp}, the $n$th syzygy $K_n$ of $\mathbf G$ is in $\text{\rm G}(R)$ and  is $\fpi$.   By the remark above, $K_n^*$ is $\fpi$.  Hence $K_n$ is in $\rG(R)$ and $\rGdim_RM<\infty$.  The result now follows from Corollary \ref{G=rGa}.
\end{proof}

As we are mainly interested in G-dimension in the context of coherent rings, and since G-dimension is equal to $\rG$-dimension for finitely presented modules over coherent rings,  we will subsequently state results only in terms of $\rG$-dimension (although some results may also hold for G-dimension).

For an $R$-module $M$ we let $\pd_RM$ denote the projective dimension of $M$.  If $M$ is $\fpi$ and $\pd_RM=n$ it is easily seen  that $M$ has a projective resolution of length $n$ consisting of finitely generated projective modules in each degree.

\begin{prop} \label{pdgd} Suppose $M$ is $\fpi$ and has finite projective dimension.  Then $\rGdim_RM=\pd_RM$.
\end{prop}
\begin{proof} Let  $n=\pd_RM$ and $m=\rG\dim_RM$.  As remarked above, there exists a projective resolution $\mathbf P$ of $M$ of length $n$ where $P_i$ is finitely generated for each $i$.   Hence, $\mathbf P$ is a $\rG$-resolution of $M$ of length $n$ and  $m\le n$.   Let $K_m$ be the $m$th syzygy of $\mathbf P$.  By Proposition \ref{GDProp}, $K_m$ is in $\rG(R)$.  It is enough to show that $K_m$ is projective.    Since
$$0\to P_n\to P_{n-1} \to \cdots \to P_{m+1}\to K_m\to 0$$
is exact and $\Ext^i_R(K,R)=0$ for all $i>0$, we have
$$0\to K_m^*\to P_{m+1}^*\to \cdots \to P^*_n\to 0$$
is also exact.  As $P_i^*$ is projective for all $i$, $K_m^*$ is projective.  Hence, $K_m^{**}\cong K_m$ is
projective.
\end{proof}

\begin{prop} \label{Gdim-ses} Suppose $0\to K\to G\to M\to 0$ is exact where $\rGdim_RM>0$ and $G$ is in $\rG(R)$.  Then $\rGdim_RK=\rGdim_RM-1$.
\end{prop}
\begin{proof} By Lemma \ref{fpi}, $K$ is $\fpi$ and $\rGdim_RK$ is defined.   It is clear that if $\rGdim_RM=\infty$ then $\rGdim_RK=\infty$, so we may assume $n=\rGdim_RM$ is positive and finite.  Clearly, $\rGdim_RK\ge n-1$.
Let $\mathbf F$ be a resolution of $K$ consisting of finitely generated free $R$-modules.  Since the composition  $\mathbf F\to G$ gives a $\rG$-resolution of $M$, the $(n-1)$st syzygy of $\mathbf F$ is in $\rG(R)$ by Proposition \ref{GDProp}.  Hence, $\rGdim_RK\le n-1$.
\end{proof}

We will also need the following:

\begin{prop}\label{Horseshoe}
Let $0\to L \to M \to N \to 0$ be an exact sequence of  $R$-modules each of which is $\fpi$.  Then the following hold: \begin{enumerate}
\item If $\rGdim_R L \leq n$ and $\rGdim_R N \leq n$, then $\rGdim_R M \leq n$.
\item If  $\rGdim_R M \leq n$ and $\rGdim_R N \leq n$, then $\rGdim_R L \leq n$.
\item If  $\rGdim_R L \leq n$ and $\rGdim_R M \leq n$, then $\rGdim_R N \leq n+1$.
\end{enumerate}
In particular, if any two of the modules has finite $\rG$-dimension, then so does the third.
\end{prop}

\begin{proof}
We prove (3).  The remaining two parts are proved similarly. Let $\mathbf F$ and $\mathbf{F'}$ be free-resolutions of $L$ and $N$, respectively, which consist of finitely generated free \(R\)-modules.  By the Horseshoe Lemma, there exists a free resolution $\mathbf{F''}$ of $M$ consisting of finitely generated free $R$-modules and chain maps $\mathbf{F}\to \mathbf{F''}$ and $\mathbf{F''}\to \mathbf{F'}$ such that
\begin{equation*}
0 \to \mathbf{F} \to \mathbf{F''} \to \mathbf{F'} \to 0
\end{equation*}
is an exact sequence of complexes.

Let $K_n$, $K_n'$, and $K_n''$ denote the $n$th syzygies of $\mathbf{F}$, $\mathbf{F'}$, and $\mathbf{F''}$, respectively.  Then the sequence
\begin{equation*}
0 \to K_n \to K''_n \to K_n'\to 0
\end{equation*}
is exact.  Since $L$ and $M$ have $\rG$-dimension at most $n$, we have that $K_n$ and $K_n''$ are in $\rG(R)$ by Proposition \ref{GDProp}.  Hence, $\rGdim_RK_n\le 1$, which implies that $\rGdim_RN\le n+1$.\end{proof}

The next result shows the restricted G-class is preserved by flat base change:

\begin{prop} \label{Gclass-flat} Let $S$ be a flat $R$-algebra and $M$ an $R$-module.
\begin{enumerate}
\item If $M$ is in $\rG(R)$ then $M\otimes_RS$ is in $\rG(S)$.
\item If $S$ is faithfully flat then the converse of (1) holds.
\end{enumerate}
\end{prop}
\begin{proof} For an $S$-module $N$, let $N^{\chk}$ denote $\Hom_S(N,S)$.    By Remark \ref{fpi-flat},  if $M$ is $\fpi^R$ then $M\otimes_RS$ is $\fpi^S$, and the converse holds if $S$ is faithfully flat.  Hence,  in the proof of (1) or its converse (under the assumption that $S$ is faithfully flat), there are natural isomorphisms $\Ext^i_R(M,R)\otimes_RS\cong \Ext^i_S(M\otimes_RS, S)$ for all $i\ge 0$.  In particular, $M^*\otimes_RS\cong (M\otimes_RS)^{\chk}$.    Again by  Remark \ref{fpi-flat}, we have that  $M^*$ is $\fpi^R$ and $(M\otimes_RS)^{\chk}$ is $\fpi^S$ in both the proof of (1) and its converse; furthermore, we have natural isomorphisms $\Ext^i_R(M^*,R)\otimes_RS\cong \Ext^i_S((M\otimes_RS)^{\chk},S)$ for all $i\ge 0$.   Thus, if $\Ext^i_R(M,R)=\Ext^i_R(M^*,R)=0$ for $i>0$, then $\Ext^i_S(M\otimes_RS,S)=\Ext^i_S((M\otimes_RS)^{\chk},S)=0$ for all $i>0$, with the converse holding if $S$ is faithfully flat.   Let $\rho:M\to M^{**}$ and $\phi:M\otimes_RS\to (M\otimes_RS)^{\chk \chk}$ be the canonical maps.  Let $K$ and $C$ be the kernel and cokernel of $\rho$,  and $K'$ and $C'$ the kernel and cokernel of $\phi$.  Then by the flatness of $S$ we have the following commutative diagram:
$$
\begin{CD}
0@>>> K\otimes_RS @>>>M\otimes_RS@>> \rho\otimes 1 >M^{**}\otimes_RS @>>> C\otimes_RS@>>> 0 \\
@. @VVV @VV=V @VV\cong V  @VVV @. \\
0@>>> K' @>>> M\otimes_RS@>>\phi> (M\otimes_RS)^{\chk \chk}@>>> C' @>>> 0.
\end{CD}
$$
If $\rho$ is an isomorphism, then so is $\rho\otimes 1$ and thus $\phi$ also.  Conversely, if $\phi$ is an isomorphism and $S$ is faithfully flat, then $K=C=0$ and thus $\rho$ is an isomorphism.
\end{proof}

Membership in the restricted G-class is a local condition, as the next result demonstrates.

\begin{prop} \label{Gclass-local}  Suppose $M$ and $M^*$ are $\fpi$.  The following are equivalent:
\begin{enumerate}
\item $M$ is in $\rG(R)$
\item $M_p$ is in $\rG(R_p)$ for all prime ideals $p$.
\item $M_m$ is in $\rG(R_m)$ for all maximal ideals $m$.
\end{enumerate}
\end{prop}
\begin{proof}
By Proposition \ref{Gclass-flat}, it is enough to prove that (3) implies (1).   As $M$ and $M^*$ are  $\fpi$,
$\Ext^i_{R_m}(M_m,R_m)\cong \Ext^i_R(M,R)\otimes_RR_m=0$ and $\Ext^i_{R_m}((M_m)^*,R_m)\cong \Ext^i_R(M^*,R)\otimes_RR_m=0$ for all $i>0$ and all maximal ideals $m$.   Furthermore, there are natural isomorphisms $(M^*)_m\cong (M_m)^{\chk}$ and  $(M^{**})_m \cong (M_m)^{\chk \chk}$, where $(-)^{\chk}=\Hom_{R_m}(-,R_m)$.   Let $\rho:M\to M^{**}$ be the canonical homomorphism and $K=\ker \rho$ and $C=\operatorname{coker}\rho$.   By hypothesis, for each maximal ideal $m$ we have  $M_m$ is in $\rG(R_m)$, which implies $K_m=C_m=0$ for all $m$.  Hence, $\rho$ is an isomorphism and $M\in \rG(R)$.
\end{proof}

\begin{cor} \label{Gdim-local} Suppose $M$ is $\fpi$.  Then
\begin{enumerate}
\item $\rGdim_RM\ge \rGdim_S M\otimes_RS$ for all flat $R$-algebras $S$.
\item $\rGdim_RM=\rGdim_S M\otimes_RS$ for all faithfully flat $R$-algebras $S$.
\item If in addition $M^*$ is $\fpi$ then
$$\rGdim_RM=\sup \{\rGdim_{R_m}{M_m}\mid m \text{ a maximal ideal of }R\}.$$
\end{enumerate}
\end{cor}
\begin{proof}
Part (1) is clear from Proposition \ref{Gclass-flat}.  For part (2), suppose $\rGdim_S M\otimes_RS =n$.  Let $\mathbf G$ be a $\rG$-resolution of $M$ and  $K$ the $n$th syzygy of $\mathbf G$.  Since $\mathbf G\otimes_RS$ is a $\rG(S)$-resolution of $M\otimes_RS$, we have that $K_n\otimes_RS$ is in $\rG(S)$ by Proposition \ref{GDProp}.   By Proposition \ref{Gclass-flat}, $K$ is in $\text{\rm G}(R)$ and hence $\Gdim_RM\le n$.  Part (3) is proved similarly using Proposition \ref{Gclass-local}.
\end{proof}

The following is a basic change of rings result.

\begin{lemma} \label{Gmodx1}  Let $M$ be in $\rG(R)$ and $x\in R$ a non-zero-divisor on $R$.  Then $M/xM$ is in $\rG(R/(x))$.
\end{lemma}
\begin{proof} Since both $M$ and $M^*$ are torsion-free, $x$ is a non-zero-divisor on $M$ and $M^*$.  Thus,  $M/xM$ and $M^*/xM^*$ are $\fpi^{R/(x)}$ by Remark \ref{fpi-nzd}.   The remainder of the proof is identical to Lemma 1.3.5 of  \cite{Christensen2000a}.
\end{proof}

\begin{lemma} \label{ext-fg}  Let $M$ be an $R$-module such that $\rGdim_RM\le n$.  Then $\Ext_R^n(M,R)$ is finitely presented.
\end{lemma}
\begin{proof} Since $\Ext^i_R(M,R)=0$ for $i>\rGdim_RM$, we may assume $n=\rGdim_RM$.  If $n=0$ the result is clear.  Suppose $n>0$ and let $0\to K\to G\to M\to 0$ be exact where $G$ is in $\rG(R)$.  By Proposition \ref{Gdim-ses}, $\rGdim_RK=n-1$.    By induction, $\Ext^{n-1}_R(K,R)$ is finitely presented.  By the exactness of
$$ \Ext^{n-1}_R(G,R)\to \Ext^{n-1}_R(K,R)\to \Ext^n_R(M,R)\to 0$$
we see that $\Ext^n_R(M,R)$ is finitely presented (e.g., Lemma \ref{fpi}).
\end{proof}

Let $J(R)$ denote the Jacobson radical of $R$.

\begin{prop} \label{Gmodx2} Let $M$ be an $R$-module which is $\fpi$ and  $x\in J(R)$  a non-zero-divisor on $M$ and $R$.    Suppose $M/xM$ is in $\rG(R/(x))$.
\begin{enumerate}
\item If $\rGdim_RM<\infty$ then $M$ is in $\rG(R)$.
\item If $R$ is coherent then $M$ is in $\rG(R)$.
\end{enumerate}
\end{prop}
\begin{proof} We prove part (1).  The proof of (2) is similar to the proof of  Lemma 1.4.4 of \cite{Christensen2000a}, where one uses  coherence to ensure the kernel and cokernel of the canonical map $M\to M^{**}$ are finitely generated so that Nakayama's lemma can be applied.
Let $\rGdim_RM=n$ and assume $n>0$.  As $x$ is a non-zero-divisor on $M$ and $R$, $\Ext^i_R(M, R/(x))\cong \Ext^i_{R/(x)}(M/xM, R/(x))=0$ for $i>0$.   Then $\Ext^n_R(M,R)$ is nonzero and finitely generated (Lemma \ref{ext-fg}).  Applying $\Hom_R(M,-)$ to the short exact sequence $0\to R\xrightarrow{x} R\to R/(x)\to 0$ yields the exact sequence
$$
\begin{CD}
\cdots @>>>\Ext^n_R(M,R) @>x>> \Ext^n_R(M,R) @>>>0,
\end{CD}
$$
which implies by Nakayama's lemma that  $\Ext^n_R(M,R)=0$, a contradiction.  Hence, $n=0$.
\end{proof}

Applying these results to restricted G-dimension, we have:

\begin{prop} \label{Gdim-modx} Let $M$ be  an $R$-module which  is $\fpi$ and $x\in R$ a non-zero-divisor on $M$ and $R$.   Then
\begin{enumerate}
\item $\rGdim_{R/(x)}M/xM \le \rGdim_RM$.
\item If $x\in J(R)$ and either $M$ has finite $\rG$-dimension or $R$ is coherent then $$\rGdim_{R/(x)}M/xM= \rGdim_RM.$$
\end{enumerate}
\end{prop}
\begin{proof}
For part (1), we may assume $\rGdim_RM=n<\infty$.  If $n=0$ the result follows by Lemma \ref{Gmodx1}.
If $n>0$ then there exists a short exact sequence $0\to K\to G\to M\to 0$ where $G$ is in $\rG(R)$ and $\rGdim_RK=n-1$.  Since $x$ is a non-zero-divisor on $M$ and $R$, the sequence
$0\to K/xK\to G/xG\to M/xM\to 0$ is exact.  By the induction hypothesis, $\rGdim_{R/(x)}K/xK\le n-1$
and $G/xG$ is in $G(R/(x))$ by the $n=0$ case.  Hence, $\rGdim_{R/(x)}M/xM\le n$.

For part (2), it is enough to show $\rGdim_RM\le \rGdim_{R/(x)}M/xM$.  Again, we may assume $\rGdim_{R/(x)}M/xM=n<\infty$.  If $n=0$ the result follows by Proposition \ref{Gmodx2}.    Otherwise,  consider a short exact sequence $0\to K\to G\to M\to 0$ where $G$ is in $\rG(R)$.  Then
$0\to K/xK\to G/xG\to M/xM\to 0$ is exact and $G/xG$ is in $\rG(R/(x))$.  By Proposition \ref{Gdim-ses} we have $\rGdim_{R/(x)}K/xK=n-1$.  Since $x$ is a non-zero-divisor on $K$ we have by the induction hypothesis that $\rGdim_RK\le n-1$.  Hence, $\rGdim_RM\le n$.
\end{proof}

We will need the following lemma in the proof of Theorem \ref{changeofrings1}:

\begin{lemma} \label{longexactseq} Let $n$ be a nonnegative integer and consider an exact sequence of $R$-modules
$$0\to M\to A_0\to A_1\to \cdots \to A_n\to 0.$$  Suppose that for $1\le i\le n$ we have $\Ext^j_R(A_i,R)=0$ for $1\le j\le i$.  Then
$$0\to A_n^*\to A_{n-1}^*\to \cdots \to A_0^*\to M^*\to 0$$
is exact.
\end{lemma}
\begin{proof}
We use induction on $n$, the cases $n=0$ and $n=1$ being trivial.  If $n>1$ let $K$ be the kernel of the map $A_{n-1}\to A_n$.  It follows easily from the short exact sequence
$0\to K\to A_{n-1}\to A_n\to 0$  that
$0\to A_n^*\to A_{n-1}^*\to K^*\to 0$ is exact and $\Ext^{i}_R(K,R)=0$ for $1\le i \le n-1$   By the induction hypothesis, we have that
$$0\to K^*\to A_{n-2}^*\to \cdots \to A_0^*\to M^*\to 0$$
is exact.  Patching the two sequences we obtain the desired result.
\end{proof}

The proof of the following important change of rings theorem is similar to the argument of \cite[Proposition 4.35]{Auslander1969a} except for a simplification (using the lemma above) which replaces a ``tedious calculation'' \cite[p. 137]{Auslander1969a}.

\begin{thm} \label{changeofrings1} Let $M$ be a nonzero  $R$-module such that $\rGdim_RM<\infty$ and suppose $x$ is a non-zero-divisor on $R$ such that $xM=0$.  Then $$\rGdim_{R/(x)}M=\rGdim_RM-1.$$
\end{thm}

\begin{proof}  We use induction on $\rGdim_RM$.  Note that since $x$ is a non-zero-divisor on $R$ and annihilates $M$, $M$ cannot be reflexive.  Hence $\rGdim_RM\ge 1$. Suppose $\rGdim_RM=1$.  Then there exists an exact sequence $0\to G_1\to G_0\to M\to 0$ such that $G_1$ and $G_0$ are in $\rG(R)$. Let $\ov{R}$ denote $R/(x)$.  For an $\ov{R}$-module $N$ let $N^{\chk}$ denote $\Hom_{\ov{R}}(N,\ov{R})$ (to distinguish this dual from $N^*=\Hom_R(N,R)$).
As $x$ is a non-zero-divisor on $R$ and $xM=0$, we have that $M^*=0$ and $\Ext^{i+1}_{R}(M,R)\cong \Ext^i_{\ov{R}}(M,\ov{R})$ for $i\ge 0$ by \cite[Lemma 18.2]{Matsumura1986a}.  Hence, $M^{\chk}\cong \Ext^1_R(M,R)$ and $\Ext^i_{\ov{R}}(M,\ov{R})=0$ for all $i>0$.  Applying $\Hom_R(-,R)$ to the short exact sequence above, we have
$$0\to G_0^{\ast} \to G_1^{\ast}\to M^{\chk}\to 0.$$
Thus, $\rGdim_RM^{\chk}=1$.   Repeating the above argument with $M^{\chk}$ in place of $M$, we get
$\Ext^i_{\ov{R}}(M^{\chk},\ov{R})=0$ for all $i>0$.  It remains to show that $M$ is reflexive as an $\ov{R}$-module.    It is readily seen by the assumptions on $x$ that $\Tor_1^R(M,R/(x))\cong M$ and $\Tor_1(G_0,R/(x))=0$.    Hence we have an exact sequence of $\ov{R}$-modules
$$0\to M\to G_1/xG_1\to G_0/xG_0\to M\to 0.$$  Note that by Lemma \ref{Gmodx1}, $G_1/xG_1$ and $G_0/xG_0$ are in $\rG(\ov{R})$.  Applying $\Hom_{\ov{R}}(-,\ov{R})$ twice to this exact sequence and using Lemma \ref{longexactseq}, we obtain the exactness of
$$0\to M^{\chk \chk}\to (G_1/xG_1)^{\chk \chk}\to (G_0/xG_0)^{\chk \chk}\to M^{\chk \chk}\to 0.$$
Since $G_0/xG_0$ and $G_1/xG_1$ are reflexive (as $\ov{R}$-modules), we see that $M$ is reflexive by the five lemma.  Hence, $\rGdim_{\ov{R}}M=0$.

Suppose $\rGdim_RM=n>1$.  Let $\phi:G\to M$ be a surjective homomorphism where $G\in \rG(R)$.
Tensoring with $R/(x)$ we get a surjection $\ov{\phi}:G/xG\to M$.  Letting $K$ be the kernel of $\ov{\phi}$
we get a short exact sequence of $\ov{R}$-modules $0\to K\to G/xG\to M\to 0$.  As $x$ is a non-zero-divisor on $G$, $0\to G\xrightarrow{x} G\to G/xG\to 0$ is exact.  Hence, $\rGdim_R G/xG=1$.  (Note that $G/xG$ is nonzero as $M$ is.)   By Proposition \ref{Horseshoe}, $\rGdim_RK<\infty$.   Moreover, as $\rGdim_RM=n>1$, it readily follows  from Proposition \ref{GDProp} that $\Ext^i_R(K,R)=0$ for $i>n-1$ and $\Ext^{n-1}_R(K,R)\neq 0$.  Hence $\rGdim_RK=n-1$.  Since $xK=0$ we have by the induction hypothesis that $\rGdim_{R/(x)}K=n-2$.   Since $G/xG$ is in $\rG(R/(x))$ (Lemma \ref{Gmodx1}), we have $\rGdim_{R/(x)}M=n-1$ by Proposition \ref{Gdim-ses}.
\end{proof}

In the coherent case, we have the same result holding if the assumption $\rGdim_RM<\infty$ is replaced with $\rGdim_{R/(x)}M/xM<\infty$ and $x\in J(R)$:

\begin{prop} \label{changeofrings2} Let $R$ be coherent, $M$ a nonzero finitely presented $R$-module, and  $x\in J(R)$  a non-zero-divisor on $R$ such that $xM=0$.  If $\rGdim_{R/(x)}M<\infty$ then $$\rGdim_{R/(x)}M=\rGdim_RM-1.$$
\end{prop}
\begin{proof}  The proof is virtually identical to the proof of Lemma 1.5.2 of \cite{Christensen2000a}, where we use part (2) of Proposition \ref{Gdim-modx} in place of \cite[Proposition 1.4.5]{Christensen2000a}.
\end{proof}

Combining Theorem \ref{changeofrings1},  Proposition \ref{changeofrings2}, and Proposition \ref{Gdim-modx}  we obtain:

\begin{cor} \label{changeofrings3} Let $R$ be coherent, $M$ a nonzero finitely presented $R$-module, and $x\in J(R)$  a non-zero-divisor on $R$.
\begin{enumerate}
\item If $xM=0$ then $\rGdim_{R/(x)}M=\rGdim_RM-1.$
\item If $x$ is a non-zero-divisor on $M$ then $\rGdim_{R}M/xM=\rGdim_{R}M+1$.
\end{enumerate}
\end{cor}

\section{The Auslander-Bridger formula for restricted Gorenstein dimension}

In this section we prove the Auslander-Bridger formula for  modules of finite restricted G-dimension over a quasi-local ring.  The proof here differs by necessity from those given in  \cite{Auslander1969a} and \cite{Christensen2000a}  to avoid arguments using the theory of associated primes.  One also needs to pass to faithfully flat extensions to make the induction argument work.

We begin with the following elementary result:

\begin{lemma} \label{dual=0} Let $R$ be a quasi-local ring with $\depth R=0$ and $M$ a finitely presented $R$-module.  Then $M=0$ if and only if $M^*=0$.
\end{lemma}

\begin{proof} Let $M$ be generated by $n$ elements and suppose $M^*=0$.  We show by induction on $n$ that $M=0$.  If $n=1$ then $M\cong R/I$ for some finitely generated ideal $I$. Since $\depth R=0$ we have $\Hom_R(R/I,R)\neq 0$ unless $I=R$ (Lemma \ref{depth-ext}).  As $M^*=\Hom_R(R/I,R)=0$ we must have $I=R$ and hence $M=0$.  If $n>1$ then there exists a submodule $M'$ of $M$ generated by $n-1$ elements and  such that $M/M'$ is cyclic.  Clearly, $M/M'$ is finitely presented and $(M/M')^*=0$.  By the $n=1$ case, we have $M/M'=0$.  Hence, $M$ is generated by $n-1$ elements and $M=0$.
\end{proof}

\begin{lemma}  \label{depth-zero} Let $R$ be a quasi-local ring with $\depth R=0$ and $M$ a nonzero $R$-module of finite $\rG$-dimension.  Then $\rGdim_RM=0$ and $\depth M=0$.
\end{lemma}

\begin{proof} By induction it suffices to prove the case when $\rGdim_RM\le 1$.   Then there exists a short exact sequence $0\to G_1\to G_0\to M\to 0$ where $G_0$ and $G_1$ are in $\rG(R)$.  Applying $\Hom_R(-,R)$ twice, we get the exact sequence
$$0\to \Ext^1_R(M,R)^*\to G_1^{**}\to G_0^{**}.$$
Since $G_0$ and $G_1$ are reflexive and the map $G_1\to G_0$ is injective, we obtain $\Ext^1_R(M,R)^*=0$.
As $\Ext^1_R(M,R)$ is finitely presented by Lemma \ref{ext-fg}, we have $\Ext^1_R(M,R)=0$ by
Lemma \ref{dual=0}.  Hence, $\rGdim_RM=0$ by Proposition \ref{GDProp}.

By way of contradiction, assume $\depth M>0$.  Let $m$ be the maximal ideal of $R$ and $S=R[x]_{mR[x]}$ where $x$ is an indeterminate over $R$.   As $S$ is faithfully flat over $R$, $M\otimes_RS$ is in
$\rG(S)$, $\depth_SS=\depth_R R=0$, and $\depth_{S} M\otimes_RS=\depth M$ (cf. Propositions \ref{Gclass-flat} and \ref{depth-basic}, part (3)).  Furthermore,  $\Depth_{S}M\otimes_RS>0$ by part (4) of Proposition \ref{depth-basic}.  Thus by passing to $S$ and resetting notation, we may assume there exists an element $x\in m$ such that $x$ is a non-zero-divisor on $M$.   Applying $\Hom_R(-,R)$ to the short exact sequence $0\to M\xrightarrow{x} M\to M/xM\to 0$ we get the exact sequence
$$0\to (M/xM)^*\to M^*\xrightarrow{x} M^* \to \Ext^1_R(M/xM,R)\to 0.$$  In particular, this sequence shows that $\Ext^1_R(M/xM, R)$ is finitely presented.    Dualizing again, we have
$$0\to \Ext^1_R(M/xM,R)^*\to M^{**}\xrightarrow{x} M^{**}$$
is exact.  Since $M\cong M^{**}$ and $x$ is a non-zero-divisor on $M$ we have $\Ext^1_R(M/xM,R)^*=0$.   By Lemma \ref{dual=0}, we see that $\Ext^1_R(M/xM,R)=0$.  From the exact sequence above, this implies $M^*=xM^*$.  By Nakayama's Lemma we have $M^*=0$, a contradiction.  Thus, $\depth M=0$.
\end{proof}

\begin{lemma} \label{Gclass-depth} Let $R$ be a quasi-local ring and $M$ a nonzero $R$-module in $\rG(R)$.  Then
$\depth R=\depth M$.
\end{lemma}
\begin{proof}  Let $m$ denote the maximal ideal of $R$.  We will induct on \(\depth R = n\).  If $n=0$ the result holds by Lemma \ref{depth-zero}.  If \(\depth R = n > 0\), as in the proof of Lemma \ref{depth-zero}  we may assume, by  passing to $R[x]_{mR[x]}$ if necessary, that there exists $x\in m$ which is a non-zero-divisor on $R$.  Then $x$ is also a non-zero-divisor on $M$ (as reflexive modules are torsion-free) and $M/xM$ is in $\rG(R/(x))$ by Lemma \ref{Gmodx1}.   Hence $\depth_{R/(x)}R/(x)=n-1$ and $\depth_{R/(x)}M/xM=\depth M-1$. Suppose first that $0 <\depth R=n<\infty$.  By induction on $n$, we obtain $\depth_{R/(x)}M/xM=n-1$ and hence $\depth M=n$.

Suppose now that $\depth R=\infty$.  We prove that $\depth M\ge n$ for all $n$. The case $n=0$ is trivial.  Assume that for all quasi-local rings $S$ with $\depth S=\infty$ and modules $N$ in $\rG(S)$, that $\depth_S N\ge n$.  As above, we may assume there exists $x\in m$ such that $x$ is a non-zero-divisor on $R$ (and hence $M$).  Then $M/xM$ is in $\rG(R/(x))$, $\depth_{R/(x)}R/(x)=\infty$, and $\depth_{R/(x)}M/xM=\depth M-1$.  By assumption, $\depth_{R/(x)}M/xM\ge n$ which implies $\depth M\ge n+1$.  \end{proof}

We now prove the Auslander-Bridger formula for restricted G-dimension:

\begin{thm} \label{AB} Let $R$ be a quasi-local ring and $M$ an nonzero $R$-module of finite $\rG$-dimension.
Then
$$\depth M+\rGdim_RM=\depth R.$$
\end{thm}
\begin{proof}
First assume $\depth R=\infty$.   If $\rGdim_RM=0$ then $\depth M=\infty$ by Lemma \ref{Gclass-depth}.
Suppose $\rGdim_RM=n>0$ and $\depth N=\infty$ for all $R$-modules $N$ such that $\rGdim_RN<n$.
Let $0\to K\to G\to M\to 0$ be exact where $G$ is in $\rG(R)$.  Then $\rGdim_RK=n-1$ and hence $\depth K=\infty$.    Since $\depth G=\infty$, $\depth M=\infty$ by part (8) of Proposition \ref{depth-basic} and the equality holds.

Assume now that $\depth R<\infty$.  We proceed by induction on $\depth R$.  If $\depth R=0$ the formula holds by Lemma \ref{depth-zero}.  Suppose $\depth R=n>0$.  Let $m$ denote the maximal ideal of $R$.  If $\depth M>0$ we may assume,  by passing to $R[x]_{mR[x]}$ if necessary,
that there exists $x\in m$ such that $x$ is a non-zero-divisor on $R$ and $M$.  (In this case, one can show  $\depth M\oplus R>0$ and we may apply part (4) of Proposition \ref{depth-basic}.)  Then $\rGdim_{R/(x)}M/xM=\rGdim_RM$ by Proposition \ref{Gdim-modx}.  Furthermore, $\depth_{R/(x)}M/xM=\depth M-1$ and $\depth_{R/(x)}R/(x)=\depth R-1$.  By the induction hypothesis , we have
$\depth_{R/(x)}M/xM + \rGdim_{R/(x)}M/xM=\depth_{R/(x)}R/(x)$.    Substituting, we see the formula holds.

Finally, assume $\depth R=n>0$ and $\depth M=0$.   Then $\rGdim_RM>0$ by Lemma \ref{Gclass-depth}.  Let $0\to K\to G\to M\to 0$ be a short exact sequence where $G$ is in $\rG(R)$.  Then $\rGdim_RK=\rGdim_RM-1$ (Proposition \ref{Gdim-ses}) and, since $\depth G=\depth R>\depth M=0$,  $\depth K=1$ (part (8) of Proposition \ref{depth-basic}).  By the $\depth M>0$ case applied to $K$, we have
$\depth K + \rGdim_RK=\depth R$.  Substituting, we again see the formula holds.

\end{proof}

As a corollary, we  get the desired result for coherent rings:

\begin{cor}\label{AB-coherent} Let $R$ be a quasi-local coherent ring and $M$ a finitely presented $R$-module of finite G-dimension.  Then
$$\depth M + \Gdim_RM = \depth R.$$
\end{cor}
\begin{proof}  As $R$ is coherent and $M$ is finitely presented, $\Gdim_RM=\rGdim_RM$ by Corollary \ref{G=rGb}.
\end{proof}

We note that Theorem \ref{AB} also generalizes (and hence gives another proof of)  \cite[Theorem 2, Chapter 6]{Northcott1976a}:

\begin{cor} Let $R$ be a quasi-local ring and $M$ an $\fpi$ module of finite projective dimension.  Then
$$\depth M+ \pd_RM = \depth R.$$
\end{cor}
\begin{proof}
By Proposition \ref{pdgd}, $\pd_RM=\rGdim_RM$.
\end{proof}

\section{Coherent Gorenstein rings}

In \cite{Bertin1971a}, J. Bertin defines a quasi-local ring to be {\it regular} if every finitely generated ideal has finite projective dimension.   A ring $R$ is said to be regular if $R_m$ is regular for every maximal ideal $m$ of $R$.   It is clear that this definition of regular reduces to the usual definition  for Noetherian local rings.   In this spirit, we propose the following definition:

\begin{defn}{\rm A quasi-local ring $R$ is called {\it Gorenstein} if every finitely generated ideal of $R$ has finite G-dimension.  In general, a ring $R$ is called {\it Gorenstein} if $R_m$ is Gorenstein for every maximal ideal $m$ of $R$.}
\end{defn}

We remark that by \cite[Theorem 4.20]{Auslander1969a}, this definition agrees with the usual notion of Gorenstein for Noetherian rings.    While most of the results and examples in this section pertain to coherent Gorenstein rings, we point out that there exist non-coherent Gorenstein rings.   We thank the referee for suggesting the following example.

A valuation ring \(R\) is \textit{almost maximal} if for every non-zero ideal \(I\), \(R/I\) is linearly compact in the discrete topology.

\begin{example} \label{noncohGor}
Let \(V\) be an almost maximal valuation domain with value  group
$\mathbb{R}$, the additive group of real numbers. (For instance, one
could take $V$ to be the ring of formal power series in one variable
over an arbitrary field $K$ with exponents in the nonnegative reals.
See section II.6 of \cite{Fuchs2001a} for details.) Let $P$ denote
the maximal ideal of $V$ and let $a$ be a nonzero element of $P$.
Then $R=V/aP$ is a non-coherent Gorenstein ring.
\end{example}

\begin{proof} Since $P$ has elements of arbitrarily small positive value, $(0:_Ra)=P/aP$ is not finitely generated.  Thus, $R$ is not coherent.  Note that the ideals of $V$ have the form $xV$ or  $xP$ for $x\in V$.  If $xV\supseteq aP$, it is easily seen that
$(aP:_VxV)=x^{-1}aP$.  Similarly, if $xP\supseteq aP$ then $(aP:_V xP)=x^{-1}aV$.    It follows that for all ideals $I$ of
$R$ one has $(0:_R(0:_R I))=I$; i.e., every ideal of $R$ is an annihilator ideal. We next show that $R$ is injective as an $R$-module.  To see this,
let  $J$ be an arbitrary ideal of $R$
generated by $\{r_{\alpha}\mid \alpha\in \Lambda \}$ where $\Lambda$ is an  index set, and let $\phi:J\to R$ be
an $R$-homomorphism.   Since for any $r\in R$ we have
 $(0:_R(0:_R r))=rR$, it follows that for each $\alpha\in \Lambda$ there exists
exists an $s_{\alpha}\in R$
such that $\phi(r_{\alpha})=s_{\alpha}r_{\alpha}$.   Since every finitely generated ideal of $R$ is principal, we get that for every finite subset $\Lambda_0$ of $\Lambda$ there exists an $s\in R$
such that $s_{\alpha}r_{\alpha}=\phi(r_{\alpha})=sr_{\alpha}$ for all $\alpha\in \Lambda_0$.    In other words,
the set of cosets $\{r_{\alpha}+(0:_Rr_{\alpha})\mid \alpha \in \Lambda\}$ has the finite intersection property.
By the linear compactness of $R$, there exists an $s\in R$ such that $\phi(r_\alpha)=sr_{\alpha}$ for all $\alpha\in \Lambda$.  Thus, $R$ is injective.  It now readily follows from the injectivity of $R$ and the fact that every ideal is an annihilator ideal  that every ideal is reflexive and every cyclic module
 is a member of $\text{\rm G}(R)$.   By induction on the number of generators, one obtains that every finitely generated $R$-module has $G$-dimension zero.  Hence, $R$ is Gorenstein.
\end{proof}

For the remainder of this section we restrict our attention to the properties of coherent Gorenstein rings.   In particular, since every finitely generated projective module is a member of the restricted G-class, we have:

\begin{prop} Let $R$ be a coherent regular ring.  Then $R$ is Gorenstein.
\end{prop}
Thus, every valuation domain is Gorenstein, as are polynomial rings in an arbitrary number of variables over a field.
We note that the property of being Gorenstein localizes:

\begin{prop} \label{Gor-local} Let $R$ be a coherent Gorenstein ring and $S$ a multiplicatively closed set.  Then $R_S$ is Gorenstein.
\end{prop}
\begin{proof} It is enough to show that if $p$ is a prime ideal of $R$ then  $R_p$ is Gorenstein. Let $I$ be a finitely generated ideal contained in $p$ and $m$ a maximal ideal containing $p$.  As $R$ is Gorenstein and coherent,  we have by Corollary \ref{Gdim-local} and Corollary \ref{G=rGb}
$$\rGdim_{R_p} I_p\le \rGdim_{R_m} I_m=\Gdim_{R_m}I_m<\infty.$$  Hence, $R_p$ is Gorenstein (and coherent).
\end{proof}

We also have:

\begin{prop} \label{Gor-fp} Let $R$ be a quasi-local coherent ring.  The following are equivalent:
\begin{enumerate}
\item $R$ is Gorenstein.
\item Every finitely presented $R$-module has finite G-dimension.
\end{enumerate}
\end{prop}
\begin{proof} An ideal $I$ is finitely generated if and only if $R/I$ is finitely presented.  Hence, (2) implies (1) is clear.  Suppose $R$ is Gorenstein and let $M$ be a finitely presented $R$-module generated by $n$ elements.  We use induction on $n$ to prove $\rGdim_RM<\infty$.  This is clear if $n\le 1$.  Suppose $n>1$.  Then there exists a submodule $M'$ of $M$ generated by $n-1$ elements such that $M/M'$
is cyclic.  As $R$ is coherent and $M$ is finitely presented, $M'$ is also finitely presented.   Hence, $\rGdim_RM'$ and $\rGdim_RM/M'$ are finite.  By Proposition \ref{Horseshoe}, $\rGdim_RM<\infty$.
\end{proof}

In \cite{Hamilton2007a} a definition for  an arbitrary  commutative ring to be Cohen-Macaulay is given which agrees with the usual definition if the ring is Noetherian.  We will show below that every coherent Gorenstein ring is Cohen-Macaulay.  Let $R$ be a ring and $\mathbf x$ denote a finite sequence $x_1,\dots,x_n$ of elements of $R$.  For an $R$-module $M$, let $\text{\v H}^i_{\mathbf x}(M)$ denote the $i$th \v Cech cohomology of $M$ with respect to $\mathbf x$ and  $\lc^i_{\mathbf x}(M)$ the $i$th local cohomology of $M$ (i.e., the $i$th right derived functor of the $(\mathbf x)$-torsion functor).    The sequence $\mathbf x$ is called {\it weakly proregular} if for all $i\ge 0$ and all $R$-modules $M$ the natural map $ \lc^i_{\mathbf x}(M)\to\text{\v H}^i_{\mathbf x}(M)$ is an isomorphism (cf. \cite{Schenzel2003a}).   A sequence $\mathbf x$ of length $n$ is called a {\it parameter sequence} if $\mathbf x$ is weakly proregular, $(\mathbf x)R\neq R$, and $\lc^n_{\mathbf x}(R)_p\neq 0$ for all prime ideals containing $(\mathbf x)$.
The sequence $\mathbf x$ is a {\it strong parameter sequence} if $x_1,\dots,x_i$ is a parameter sequence for all $1\le i\le n$.   A ring $R$ is called {\it Cohen-Macaulay} if every strong parameter sequence on $R$ is a regular sequence.    It is not known if this property localizes;  thus, we say that $R$ is {\it locally Cohen-Macaulay} if $R_p$ is Cohen-Macaulay for all prime ideals $p$ of $R$.  It is easily seen that locally Cohen-Macaulay rings are Cohen-Macaulay; see \cite{Hamilton2007a} for details.

\begin{prop} \label{Gor-CM} Let $R$ be a coherent Gorenstein ring.  Then $R$ is locally Cohen-Macaulay.
\end{prop}
\begin{proof}  Since $R_p$ is coherent Gorenstein for all primes $p$ of $R$, it suffices to prove that a coherent quasi-local Gorenstein ring is Cohen-Macaulay.  Let $\mathbf x=x_1,\dots,x_n$ be a strong parameter sequence of $R$.  We proceed by induction on $n$ to show that $\mathbf x$ is a regular sequence.  Let $\mathbf x'$ denote the sequence $x_1,\dots,x_{n-1}$.  (In the case $n=1$, $\mathbf x'$ denotes the empty sequence, which generates the zero ideal.)  By way of contradiction, we assume $\mathbf x'$ is a regular sequence but $x_n$ is a zero-divisor on $R/(\mathbf x')$.    Then there exists a prime $p$ of $R$ such that $x_n\in p$ and
$\depth R_p/(\mathbf x')R_p=0$  \cite[Lemma 2.8]{Hamilton2007a}.  Hence, $\depth R_p=n-1$.  By localizing at $p$, we can assume $R$ is a quasi-local coherent Gorenstein ring with $\depth R=n-1$ and
$\mathbf x$ is a parameter sequence of length $n$.  In particular, $H^{n}_{\mathbf x}(R)\neq 0$.
As $R$ is coherent Gorenstein and $(\mathbf x)^{t}$ is finitely generated for all $t$, we have $\rGdim_R R/(\mathbf x)^t<\infty$ for all $t$.  By Theorem \ref{AB}, $\rGdim_RR/(\mathbf x)^t\le \depth R=n-1$ for all $t$, and thus $\Ext^n_R(R/(\mathbf x)^t,R)=0$ for all $t$.   But then
$$H^n_{\mathbf x}(R)\cong \directlimit{t}{\Ext^n_R(R/(\mathbf x)^t,R)}=0,$$
a contradiction.  Hence, $\mathbf x$ is a regular sequence and $R$ is Cohen-Macaulay.
\end{proof}

We next show that the Gorenstein property is preserved by passing to (and lifting from) a quotient modulo a non-zero-divisor.   This certainly does not hold for regularity.  We note also that, in contrast to the Noetherian case, there are examples of quasi-local Cohen-Macaulay rings which do not remain Cohen-Macaulay modulo a non-zero-divisor.  (See Example 4.9 of \cite{Hamilton2007a}.)

\begin{thm} \label{Gor-modx} Let $R$ be a quasi-local coherent ring with maximal ideal $m$ and $x\in m$ a non-zero-divisor on $R$.   Then $R$ is Gorenstein if and only if $R/(x)$ is Gorenstein.
\end{thm}
\begin{proof} Assume $R$ is Gorenstein.  Let $J$ be a finitely generated ideal of $R/(x)$.  Then $J=I/(x)$ for some finitely generated ideal $I$ of $R$ containing $(x)$.  As $R$ is (coherent) Gorenstein, $\rGdim R/I<\infty$.  By Theorem \ref{changeofrings1}, $\rGdim_{R/(x)}R/I<\infty$, which implies $\rGdim_{R/(x)}J<\infty$.  Thus, $R/(x)$ is Gorenstein.

Conversely, assume $R/(x)$ is Gorenstein.  Let $M$ be a finitely presented $R$-module.  Assume first that $x$ is a non-zero-divisor on $M$.  Then $M/xM$ is a finitely presented $R/(x)$-module and by Proposition \ref{Gor-fp}, $\rGdim_{R/(x)}M/xM<\infty$.  By Proposition \ref{Gdim-modx}, $\rGdim_RM<\infty$.  If $x$ is a zero-divisor on $M$ let $0\to K\to G\to M\to 0$ be a short exact sequence where $G$ is in $\text{\rm G}(R)$.  As $R$ is coherent, $K$ is finitely presented.  Furthermore, $x$ is a non-zero-divisor on $G$ and hence on $K$.  Thus, $\rGdim_RK<\infty$.  Hence, $\rGdim_RM<\infty$.
\end{proof}

Hence, for example,  if $V$ is a valuation domain and $x\in V$ is a non-unit non-zero-divisor, then $V/xV$ is Gorenstein.

We now aim to prove that the Gorenstein property passes to finitely generated polynomial ring extensions, assuming coherence is preserved.  To do this,  we first need a folklore result which can be traced at least as far back as 1966 (\cite{Jensen1966a}; see also \cite{Vasconcelos1976a}).  The proof here is of a different style than the one found in the above references,  although the argument is  essentially the same.

\begin{lemma} \label{R[x]-ses} Let $R$ be a ring, $x$ an indeterminate over $R$, and $M$ an $R[x]$-module.    There exists a short exact sequence of $R[x]$-modules
$$0\to R[x]\otimes_RM\to  R[x]\otimes_RM\to M\to 0.$$
\end{lemma}
\begin{proof} Let $t$ be an indeterminate over $R[x]$ and let $t$ act on $M$ via $t\cdot m:=xm$ for all $m\in M$.  In this way one can consider $M$  as an $R[x,t]$-module.
Consider the short exact sequence of $R[x,t]$-modules
$$0\to R[x,t]\xrightarrow{x-t} R[x,t] \to R[t]\to 0.$$
Applying $-\otimes_{R[t]}M$, we get the short exact sequence of $R[x,t]$-modules
$$0\to R[x,t]\otimes_{R[t]}M\to R[x,t]\otimes_{R[t]}M\to R[t]\otimes_{R[t]}M\to 0.$$
Now, $R[x,t]\otimes_{R[t]}M\cong (R[x]\otimes_RR[t])\otimes_{R[t]}M\cong R[x]\otimes_RM$ as $R[x,t]$-modules.  Thus, we have a short exact sequence of $R[x,t]$-modules
$$0\to R[x]\otimes_RM\to R[x]\otimes_RM\to M\to 0.$$
Restricting scalars to $R[x]$ we get the desired result.
\end{proof}

This leads to another change of rings result for restricted Gorenstein dimension:

\begin{cor} \label{Gor-cor}Let $R$ be a ring, $x$ an indeterminate over $R$, and $M$ an $R[x]$-module which is $\fpi^R$.  Then $\rGdim_{R[x]}M\le \rGdim_RM+1$.
\end{cor}
\begin{proof}  It is easily seen that $M$ is also $\fpi^{R[x]}$ and hence $\rGdim_{R[x]}M$ is defined.  Clearly, we may assume $\rGdim_RM<\infty$.  Since $R[x]$ is faithfully flat as an $R$-module, $\rGdim_RM=\rGdim_{R[x]}R[x]\otimes_RM$ by Proposition \ref{Gdim-local}.  The result now follows from Lemma \ref{R[x]-ses} and Proposition \ref{Horseshoe}.
\end{proof}

\begin{thm} Let $R[x]$ be a polynomial ring over $R$ and assume $R[x]$ is coherent.  Then $R$ is Gorenstein if and only if $R[x]$ is Gorenstein.
\end{thm}
\begin{proof}
Suppose $R[x]$ is Gorenstein.  As $x$ is a non-zero-divisor on $R$, $R\cong R[x]/(x)$ is Gorenstein (and coherent) by Theorem \ref{Gor-modx}.  Conversely, suppose $R$ is Gorenstein and let $P$ be a maximal ideal of $R[x]$.  By localizing at $q=P\cap R$, we can assume $R$  is quasi-local coherent Gorenstein (by Proposition \ref{Gor-local})   with maximal ideal $m$ and $P\cap R=m$.  Now, $P=(m,f)R[x]$ for some monic polynomial  $f\in R[x]$.
Then $R[x]/fR[x]$ is a free $R$-module of finite rank.    Let $J$ be a finitely generated ideal of $R[x]_{P}$.
Then $J=I_{P}$ for some finitely generated (hence, finitely presented) ideal $I$ of $R[x]$.    Consequently, $I/fI$
is a finitely presented $R[x]/fR[x]$-module, and hence finitely presented as an $R$-module as well.  Thus, $\rGdim_RI/fI<\infty$.  By  Corollary \ref{Gor-cor}, $\rGdim_{R[x]}I/fI<\infty$.  Localizing, we have $\rGdim_{R[x]_{P}}J/fJ<\infty$.  Since $f\in P$ is a non-zero-divisor on both $R[x]_P$ and $J$, we have that $\rGdim_{R[x]_P}J<\infty$ by Corollary \ref{changeofrings3}.  Hence, $R[x]_P$ is Gorenstein.
\end{proof}

It is known that the direct limit of a flat family of  coherent regular rings is coherent regular (e.g., \cite[Theorem 6.2.2]{Glaz1989a}).  The analogous result holds for Gorenstein rings:

\begin{prop} \label{directlimit}  Let $\{R_i\}_{i\in \Lambda}$ be a direct system of commutative rings over a directed index set $\Lambda$.
Suppose each $R_i$ is a coherent Gorenstein ring and the maps $R_i\to R_j$  are flat for all $i\le j$.
Then $\directlimit{}{R_i}$ is a coherent Gorenstein ring.
\end{prop}
\begin{proof} Let $S=\directlimit{}{R_i}$ and $Q$ a prime ideal of $S$.  Since $S_Q\cong \directlimit{}{(R_i){q_i}}$, where $q_i=Q\cap R_i$, we may assume each $R_i$ is quasi-local.   Thus, $R_i\to S$ is faithfully flat for all $i$.  Let $I$ be a finitely generated ideal of $S$.  Then there exists a $j\in \Lambda$ and a  finitely generated ideal $J$ of $R_j$ such that $JS=I$.  Since $R_j$ is coherent Gorenstein and quasi-local, $\rGdim_{R_i}R_j/J<\infty$.  Since $S$ is faithfully flat over $R_j$, by Corollary \ref{Gdim-local}, $\rGdim_S S/I<\infty$.  Hence, $S$ is Gorenstein.  Note that the argument also shows that $S$ is coherent.
\end{proof}

As an application, we have the following:

\begin{example}Let $S=k[x_1,x_2,x_3,\dots]/(x_1^2, x_2^2,x_3^2,\dots)$, where $k$ is a field and the $x_i$ are variables.  Then $S\cong \directlimit{}{R_i}$
where for $i\in \mathbb N$,  $R_i=k[x_1,x_2,\dots,x_i]/(x_1^2,x_2^2,\dots, x_i^2)$.  Since each $R_i$ is a coherent Gorenstein ring and the maps $R_i\to R_{i+1}$ are flat for all $i$, $S$ is Gorenstein by Proposition \ref{directlimit}.
\end{example}

It is well-known that a Noetherian local ring is Gorenstein if and only if the ring has finite injective dimension.  This characterization can be generalized to quasi-local coherent rings using the notion of FP-injective dimension as defined by Stenstr\"om \cite{Stenstrom1970a}:

\begin{defn}{\rm Let $M$ be an $R$-module.  The {\it FP-injective dimension} of $M$ is defined by
$$\fpid_RM:=\inf\{n\ge 0\mid \Ext^{n+1}_R(N,M)=0 \text{ for every finitely presented $R$-module }N\},$$  where the infimum of the empty set is defined to be $\infty$.  If $\fpid_RM=0$ we say $M$ is {\it FP-injective}.}
\end{defn}

\begin{prop} \label{fpid} Let $R$ be a coherent ring, $M$ an $R$-module, and $n$ a nonnegative integer.  The following conditions are equivalent:
\begin{enumerate}
\item $\fpid_RM\le n$.
\item $\Ext^i_R(N,M)=0$ for all $i>n$ and all finitely presented $R$-modules $N$.
\item $\Ext^{n+1}_R(R/I,M)=0$ for all finitely generated ideals $I$ of $R$.
\item For every exact sequence
$$0\to M\to E^0\to E^1\to \cdots \to E^{n-1}\to E^n\to 0$$
such that $E^i$ is FP-injective for $0\le i\le n-1$, we have $E^n$ is FP-injective.
\end{enumerate}
\end{prop}
\begin{proof}  See Lemma 3.1 of \cite{Stenstrom1970a}.
\end{proof}

Combining this result with Proposition \ref{depth-coherent} we have:

\begin{cor} \label{inequality} Let $R$ be a coherent quasi-local ring with maximal ideal $m$ and $M$ an $R$-module such that $mM\neq M$.  Then $\depth M\le \fpid_RM$.
\end{cor}

We now give a characterization of quasi-local Gorenstein rings of finite depth in terms of FP-injective dimension.  (For a proof of this result in the Noetherian case, see Theorem 4.20 of \cite{Auslander1969a}.)

\begin{thm} \label{Gor=fpid} Let $R$ be a quasi-local coherent ring, and $n$ a nonnegative integer.   The following conditions are equivalent:
\begin{enumerate}
\item $\fpid_RR=n$.
\item $R$ is Gorenstein and $\depth R=n$.
\end{enumerate}
\end{thm}
\begin{proof} Suppose $R$ is Gorenstein and $\depth R=n$.  Let $M$ be a finitely presented $R$-module.  By Theorem \ref{AB}, $\rGdim_RM\le n$ and hence $\Ext^i_R(M,R)=0$ for all $i>n$ by Proposition \ref{GDProp}.   On the other hand, $\Ext^n_R(R/I,R)\neq 0$
for some finitely generated ideal $I$ of $R$ by Proposition \ref{depth-coherent}.  Therefore, $\fpid_RR=n$.

Conversely, suppose $\fpid_RR=n$.  By Corollary \ref{inequality}, $\depth R=m\le n$.    If  $R$ is Gorenstein then $m=n$ as  we have shown (2) implies (1).    Let $M$ be a finitely presented $R$-module and consider
$$0\to K\to F_{n-1}\to \cdots \to F_1\to F_0\to M\to 0$$
where each $F_i$ is a finitely generated free $R$-module.  Since $\Ext^i_R(M,R)=0$ for all $i>n$, it is easily seen that $\Ext^i_R(K,R)=0$ for all $i>0$.  It suffices to show that $K$ is in $\rG(R)$.  In fact, we will show that for any finitely presented $R$-module $N$ such that $\Ext^i_R(N,R)=0$ for all $i>0$, one has $N\in \rG(R)$.  Let
$$0\to C\to G_{n-1}\to G_{n-2}\to \cdots \to G_1\to G_0\to N\to 0$$
be exact where $G_i$ is a finitely generated free $R$-module for all $i$.   By Lemma \ref{longexactseq}  the sequence
$$0\to N^*\to G_0^*\to G_1^*\to \cdots \to G_{n-1}^*\to C^*\to 0$$
is exact.
As $C^*$ is finitely presented (coherence) and $\fpid_RR=n$, $\Ext^i_R(C^*,R)=0$ for $i>n$.  Hence (by the same argument as for $K$ above),
$\Ext^i_R(N^*,R)=0$ for $i>0$.  Now consider the short exact sequence $0\to L\to G_0\to N\to 0$.   Then $\Ext^i_R(L,R)=0$ for all $i>0$ and hence $\Ext^i_R(L^*,R)=0$ for $i>0$ as well (using $L$ in place of $N$).  Applying $\Hom_R(-,R)$ twice, one obtains the exactness of $$0\to L^{**}\to G_0^{**}\to N^{**}\to 0.$$  Since the canonical map $G_0\to G_{0}^{**}$ is an isomorphism, the canonical  map $N\to N^{**}$ is surjective.  With $L$ in place of $N$, we also obtain the map $L\to L^{**}$ is surjective.  Hence, the map $N\to N^{**}$ is an isomorphism (by the snake lemma) and $N$ is in $\rG(R)$. Thus, $R$ is Gorenstein.
\end{proof}

We remark that a version of Theorem \ref{Gor=fpid} (without reference to depth) was proved in \cite[Theorem 7]{Ding1996a} using very different methods.  We also note that a polynomial ring over a field in infinitely many variables localized at a maximal ideal is a coherent quasi-local regular (hence Gorenstein) ring with infinite FP-injective dimension.

Another equivalent characterization for a local ring $R$ to be Gorenstein is that $R$ be Cohen-Macaulay and some (equivalently, every) system of parameters generate an irreducible ideal.  We have already seen that coherent Gorenstein rings are Cohen-Macaulay.  Additionally, we have the following:

\begin{prop} Let $R$ be a quasi-local coherent  Gorenstein ring with $\depth R=n<\infty$.  Then every $n$-generated ideal generated by a  regular sequence is irreducible.   \end{prop}
\begin{proof}  Since the Gorenstein property is preserved modulo a regular sequence, it suffices to prove the case $n=0$.   Suppose $(x)\cap (y)=(0)$ for some $x,y\in R$.  Then $(0:_R(x)\cap (y))=R$.  Consider the exact sequence
$$0\to R/((x)\cap (y))\to R/(x)\oplus R/(y)\to R/(x,y)\to 0.$$
Since $R$ is coherent Gorenstein of depth zero, $\Ext^1_R(R/(x,y),R)=0$.  Hence, applying $\Hom_R(-,R)$ we obtain
the exactness of
$$0\to (0:_R(x,y))\to (0:_Rx)\oplus (0:_Ry)\to (0:_R(x)\cap (y))\to 0.$$
Thus, $R=(0:_R(x)\cap (y))=(0:_Rx)+(0:_Ry)$.  As $R$ is quasi-local, this implies $x=0$ or $y=0$.
\end{proof}

We remark that in the above proposition, regular sequences of length $n=\depth R$ may not exist.  However, one can pass to a faithfully flat extension (assuming coherence is preserved) to obtain such  sequences.

\bibliography{citeGorenstein}

\providecommand{\bysame}{\leavevmode\hbox to3em{\hrulefill}\thinspace}
\providecommand{\MR}{\relax\ifhmode\unskip\space\fi MR }
\providecommand{\MRhref}[2]{%
  \href{http://www.ams.org/mathscinet-getitem?mr=#1}{#2}
}
\providecommand{\href}[2]{#2}
\begin{thebibliography}{10}

\bibitem{Auslander1969a}
Maurice Auslander and Mark Bridger, \emph{Stable module theory}, Memoirs of the
  American Mathematical Society, No. 94, American Mathematical Society,
  Providence, R.I., 1969. \MR{MR0269685 (42 \#4580)}

\bibitem{Barger1972a}
S.~Floyd Barger, \emph{A theory of grade for commutative rings}, Proc. Amer.
  Math. Soc. \textbf{36} (1972), 365--368. \MR{MR0308106 (46 \#7221)}

\bibitem{Bertin1971a}
Jos{\'e} Bertin, \emph{Anneaux coh\'erents r\'eguliers}, C. R. Acad. Sci. Paris
  S\'er. A-B \textbf{273} (1971), A590--A591. \MR{MR0288116 (44 \#5314)}

\bibitem{Bieri1976a}
Robert Bieri, \emph{Homological dimension of discrete groups}, Mathematics
  Department, Queen Mary College, London, 1976, Queen Mary College Mathematics
  Notes. \MR{MR0466344 (57 \#6224)}

\bibitem{Bruns1993a}
Winfried Bruns and J{\"u}rgen Herzog, \emph{Cohen-{M}acaulay rings}, Cambridge
  Studies in Advanced Mathematics, vol.~39, Cambridge University Press,
  Cambridge, 1993. \MR{MR1251956 (95h:13020)}

\bibitem{Christensen2000a}
Lars~Winther Christensen, \emph{Gorenstein dimensions}, Lecture Notes in
  Mathematics, vol. 1747, Springer-Verlag, Berlin, 2000. \MR{MR1799866
  (2002e:13032)}

\bibitem{Ding1996a}
Nanqing Ding and Jialong Chen, \emph{Coherent rings with finite
  self-$fp$-injective dimension}, Comm. Algebra \textbf{24} (1996), no.~9,
  2963--2980. \MR{MR1396867 (97f:16016)}

\bibitem{Enochs1995a}
Edgar~E. Enochs and Overtoun M.~G. Jenda, \emph{Gorenstein injective and
  projective modules}, Math. Z. \textbf{220} (1995), no.~4, 611--633.
  \MR{MR1363858 (97c:16011)}

\bibitem{Fuchs2001a}
L{\'a}szl{\'o} Fuchs and Luigi Salce, \emph{Modules over non-{N}oetherian
  domains}, Mathematical Surveys and Monographs, vol.~84, American Mathematical
  Society, Providence, RI, 2001. \MR{MR1794715 (2001i:13002)}

\bibitem{Glaz1989a}
Sarah Glaz, \emph{Commutative coherent rings}, Lecture Notes in Mathematics,
  vol. 1371, Springer-Verlag, Berlin, 1989. \MR{MR999133 (90f:13001)}

\bibitem{Glaz1992a}
\bysame, \emph{Commutative coherent rings: historical perspective and current
  developments}, Nieuw Arch. Wisk. (4) \textbf{10} (1992), no.~1-2, 37--56.
  \MR{MR1187898 (93j:13024)}

\bibitem{Glaz1994a}
\bysame, \emph{Coherence, regularity and homological dimensions of commutative
  fixed rings}, Commutative algebra (Trieste, 1992), World Sci. Publ., River
  Edge, NJ, 1994, pp.~89--106. \MR{MR1421079 (97f:13012)}

\bibitem{Hamilton2007a}
Tracy~Dawn Hamilton and Thomas Marley, \emph{Non-{N}oetherian
  {C}ohen-{M}acaulay rings}, J. Algebra \textbf{307} (2007), no.~1, 343--360.
  \MR{MR2278059}

\bibitem{Hochster1974a}
M.~Hochster, \emph{Grade-sensitive modules and perfect modules}, Proc. London
  Math. Soc. (3) \textbf{29} (1974), 55--76. \MR{MR0374118 (51 \#10318)}

\bibitem{Jensen1966a}
Chr.~U. Jensen, \emph{On homological dimensions of rings with countably
  generated ideals}, Math. Scand. \textbf{18} (1966), 97--105. \MR{MR0207796
  (34 \#7611)}

\bibitem{Matsumura1986a}
Hideyuki Matsumura, \emph{Commutative ring theory}, Cambridge Studies in
  Advanced Mathematics, vol.~8, Cambridge University Press, Cambridge, 1986,
  Translated from the Japanese by M. Reid. \MR{MR879273 (88h:13001)}

\bibitem{Miller2008a}
Liva Miller, \emph{{A theory of non-Noetherian Gorenstein rings}}, Ph.D.
  thesis, University of Nebraska-Lincoln, 2008.

\bibitem{Northcott1976a}
D.~G. Northcott, \emph{Finite free resolutions}, Cambridge University Press,
  Cambridge, 1976, Cambridge Tracts in Mathematics, No. 71. \MR{MR0460383 (57
  \#377)}

\bibitem{Schenzel2003a}
Peter Schenzel, \emph{Proregular sequences, local cohomology, and completion},
  Math. Scand. \textbf{92} (2003), no.~2, 161--180. \MR{MR1973941
  (2004f:13023)}

\bibitem{Stenstrom1970a}
Bo~Stenstr{\"o}m, \emph{Coherent rings and {$FP$}-injective modules}, J. London
  Math. Soc. (2) \textbf{2} (1970), 323--329. \MR{MR0258888 (41 \#3533)}

\bibitem{Vasconcelos1976a}
Wolmer~V. Vasconcelos, \emph{The rings of dimension two}, Marcel Dekker Inc.,
  New York, 1976, Lecture Notes in Pure and Applied Mathematics, Vol. 22.
  \MR{MR0427290 (55 \#324)}

\end{thebibliography}

\end{document}